\documentclass[11pt]{amsart}
\usepackage{amsfonts}
\usepackage{amssymb}
\usepackage{graphicx}%
\usepackage{tikz}
\usetikzlibrary{shapes.geometric, arrows}
\tikzstyle{io} = [rectangle, rounded corners, minimum width=2cm, minimum height=1cm,text centered, draw=black, fill=white]
\tikzstyle{arrow} = [thick,->,>=stealth]

\usepackage{mathtools}
\usepackage{color} 
\usepackage[utf8]{inputenc} 

\usepackage{amscd}
\usepackage{mathrsfs}
\newtheorem{theorem}{Theorem}[section]
\newtheorem{lemma}[theorem]{Lemma}
\newtheorem{corollary}[theorem]{Corollary}
\newtheorem{definition}[theorem]{Definition}

\newtheorem{remark}[theorem]{Remark}
\newtheorem{notation}[theorem]{Notation}

\numberwithin{equation}{section}

\renewenvironment{proof}[1][\proofname]{%
  \par\pushQED{\qed}\normalfont%
  \trivlist\item[\hskip\labelsep\bfseries#1{.}]%
}{%
  \popQED
}

\def\DO{\mathcal D}
\def\RE{\mathbb R}

\def\N{\mathbb N}

\def\C{\mathcal C}
\def\L{\mathcal L}

\def\A{\mathcal A}

\def\G{\mathcal G}

\def\cp{\circledast}

\begin{document}

\title[Algebras of Schwartz distributions]{An existence and uniqueness result about algebras of Schwartz distributions}

\author{Nuno Costa Dias}

\author{Cristina Jorge}

\author{Jo\~{a}o Nuno Prata}

 \maketitle

\begin{abstract}
We prove that there exists essentially one {\it minimal} differential algebra of
distributions $\A$, satisfying all the properties stated in the
Schwartz impossibility result [L. Schwartz, Sur l’impossibilité de la multiplication des distributions, 1954], and such that
$\C_p^{\infty} \subseteq \A  \subseteq \DO' $ (where
$\C_p^{\infty}$ is the set of piecewise smooth functions and
$\DO'$ is the set of Schwartz distributions over $\RE$). This algebra is
endowed with a multiplicative product of distributions, which is a generalization of the product
defined in [N.C.Dias, J.N.Prata, A multiplicative product of distributions and a class of ordinary differential equations with
distributional coefficients, 2009]. If the algebra is not minimal, but satisfies the previous conditions, is closed under anti-differentiation and the dual product by smooth functions, and the distributional product is continuous at zero then it is necessarily an extension of $\A$.   
\end{abstract}



\section{Introduction}

The Schwartz famous impossibility result \cite{Sch54} states that:

\begin{theorem} \label{SchwartzT}
There is no associative algebra ($\G,+,\circledast $) satisfying the following properties:
\begin{enumerate}
	\item [(A1)] The space of Schwartz distributions $\DO'$ over $\RE$ is linearly embedded into ${\mathcal G}$.
	\item [(A2)] The function $f(x)=1$ is the identity in $\G$.
	\item [(A3)] There exists a linear derivative operator $D:\G\rightarrow \G$ that:
	\begin{enumerate}
	\item [(a)] satisfies the Leibniz rule, and
	\item [(b)] coincides with the usual distributional derivative $D_x$ in $\G \cap \DO'$.
	\end{enumerate}
	\item [(A4)] The multiplication $\circledast$ coincides with the usual product of functions in $(\G\cap\C)\times(\G \cap\C)$, $\C$ is the space of continuous functions over $\RE$.
\end{enumerate}
\end{theorem}

If we replace (A1) by: 
\begin{enumerate}
\item [(A1')] \hspace{0.0cm} $\C^\infty_p\subseteq \G\subseteq \DO'$
\end{enumerate}
where $\C^{\infty}_p$ is the set of piecewise smooth functions, then: 

\begin{theorem} \label{DP}
Let $\A \equiv \cup_{i=0}^\infty D^i_x [\C^{\infty}_p]$ be the minimal space containing $\C_p^\infty$ and closed for $D_x$, and let $*_M$, $M \subseteq \RE$, be the multiplicative product of distributions given in Definition \ref{Gprod1}.
 The family of associative algebras ($\A$,+,$*_M$), $M \subseteq \RE$ satisfies the conditions (A1') and (A2)-(A4).
\end{theorem}

The products $*_M$, $M \subseteq \RE$ (cf. Definition \ref{Gprod1}) are extensions (to the case of possible intersecting singular supports) of the product of distributions with disjoint singular supports presented by H\"ormander in [pag.55 \cite{Hor83}]. Theorem \ref{DP} was proved by two of us for the case $M=\RE$ in \cite{DP09}, and will be (easily) extended to the general case $M\subseteq \RE$ in section 2.3.

In this paper we want to study the related problem of whether the associative algebras ($\A,+,*_M$) are unique, i.e. the only ones satisfying the conditions (A1') and (A2)-(A4).

Let us introduce the following notation: We say that in $\G$ the product $\cp$ by smooth functions is continuous (or simply that $\cp$ is {\it partially continuous}) at $F\in \G$ iff for every $\xi\in \C^{\infty}$, and every sequence $(F_n)_{n\in \N}$, $F_n
\overset{\DO'}{\longrightarrow} F$ in $\G$, we have:
$
\xi\cp F_n \overset{\DO'}{\longrightarrow} \xi \cp F$, and  
$F_n \cp \xi \overset{\DO'}{\longrightarrow} F \cp \xi $,
where $\overset{\DO'}{\longrightarrow}$ denotes convergence in distribution sense. 

We note that the dual product and the family of products $*_M$ (defined in $\A$) are all partially continuous (cf. Theorem 2.11(vi)). We also remark that if $\cp$ (defined in $\G$) is partially continuous at zero then it is partially continuous everywhere in $\G$ (because $\cp$ is bilinear, and $\G$ is a vector space). 


Our results are summarized in the following Theorem and Corollary:

\vspace{0.25cm}

\noindent\textbf{Main Theorem.}  {\it Let 
($\G,+,\circledast$) be an associative algebra of distributions that satisfies the conditions (A1'), (A2)-(A4) given above, and
\begin{enumerate}
\item [(A5.1)] Every $F \in \G$ is locally a finite order derivative of some $G \in \G \cap \C$,
\item [(A5.2)] The product $\cp$ is partially continuous at zero; 
\end{enumerate}
then $\A \subseteq \G$ and the restriction of $\circledast$ to $\A$ is given by $*_M$ for some $M \subseteq \RE$. In other words, ($\A,+,*_M$) is a subalgebra of ($\G,+,\circledast$).  
}

\begin{remark} \label{remark1}

Notice that every $F \in \DO'$ is locally a finite order derivative of some continuous function $G \in \C$ (cf. Theorem 3.4.2, \cite{Zemanian}). The condition (A5.1) adds the requirement that if $F\in \G$ then also $G \in \G$. This condition can be replaced by (cf. Remark \ref{remarkak}):
\begin{enumerate}
\item [(A5.1')] \hspace{0.0cm} Anti-differentiation and the dual product by smooth functions are inner operations in $\G$.
\end{enumerate}
The conditions (A5.1) and (A5.1') are both satisfied by 
$\G = \DO'$ and $\G =\A$. 

We also note that the conditions (A5.1) and (A5.2) can be replaced by the single, stronger condition (cf. Remark \ref{remarkA6}):
\begin{enumerate}
\item [(A6)] Every $F \in \G$ is globally a finite order derivative of some $G \in \G \cap \C$;
\end{enumerate}
which is satisfied by $\G=\A$ and by $\G=\DO'(\Omega)$ for arbitrary compact sets $\Omega \subset \RE$. We will see that (A6) implies (A5.2) (and, of course, also (A5.1)). The imposition of one of the conditions (A5.1,A5.2), (A5.1',A5.2) or (A6) is required for the proof of Theorem \ref{ak} which is an important intermediate result in the proof of the Main Theorem.
\end{remark}

\begin{corollary}\label{corollary1}
For $\A$ as defined above, let ($\A,+,\circledast$) be an associative algebra satisfying the conditions (A2)-(A4). Then $\circledast =*_M$ for some $M \subseteq \RE$.  	
\end{corollary}
The proof of this Corollary is straightforward: since the space $\A$ satisfies (A1') and (A6), and thus (cf. Remark \ref{remark1}) also (A5.1) and (A5.2), it follows from the Main Theorem that $\circledast =*_M$ for some $M\subseteq \RE$.  

The problem of proving the uniqueness of the algebras ($\A,+,*_M$) was considered before in an article by B. Fuchssteiner that was published in Mathematische Annalen \cite{Fuch68} (see also \cite{Fuch84}) and recently reviewed in the Ph.D thesis {\cite{Trenn}}. The main result of  {\cite{Fuch68}} is basically our Corollary \ref{corollary1}. Unfortunately, the paper {\cite{Fuch68}} is not so well-known and came to our knowledge only after we have concluded our own proof of the uniqueness result. In spite of the obvious intersection with the results of {\cite{Fuch68}}, we have decided to write down our own results because: (i) our proof is different and, in our view, simpler than the one presented in {\cite{Fuch68,Trenn}}; and (ii) our results are more general, because they do not apply only to the space $\A$, but instead to the family $\G \subseteq \DO'$. 
In practice this means that we do not impose the restriction that the product $\circledast$ is an inner operation in $\A$; instead we prove that this  is a consequence of the properties (A1'), (A2)-(A5) for a general space $\G \subseteq \DO'$.  

Finally, we remark that the algebras $(\A , +,*_M)$ provide an interesting setting to obtain intrinsic formulations (i.e. defined within the space of Schwartz distributions) for some classes of differential operators and differential equations with singular coefficients. This approach has been  explored namely for Schr\"odinger operators with point interactions and for ODEs with singular coefficients \cite{DJP16,DJP19,DJP20}. It yields a formulation which is more general than the ones based on other intrinsic products like the model products \cite{A12,HO07,Obe92}, and alternative to the non-intrinsic formulations like the ones in terms of Colombeau generalized functions \cite{Col84,Col92,GKO01,HKO13,Obe92,Ros87}. 


In the next section we study the main properties of the product $*_M$ and show that for all $M \subseteq \RE$, the algebras $(\A,+,*_M)$ satisfy the conditions in Theorem \ref{DP}. In section 3 we prove the Main Theorem.

\begin{notation}
Spaces of functions or distributions over $\RE$ are denoted by calligraphic capital letters $\A$, $\C$, $\DO$, $\DO'$,.... 

Capital roman letters $F$, $G$ and $J$ denote general distributions; $\phi$, $\psi$ and $\xi$ are smooth functions; and $f$, $g$ and $h$ are locally integrable functions or regular distributions (we normally use the same notation for both objects). If we need to be more specific, we use the subscript $\DO'$ for regular distributions; for instance $f_{\DO'}$ is the regular distribution associated to the locally integrable function $f$.   

The characteristic function of $\Omega \subseteq \RE$ is written $\chi_\Omega$; the Heaviside step function is $H=\chi_{\RE^+}$. As usual $\delta_{x}$ is the Dirac measure with support at $x$; if $x=0$ we write only $\delta$.
 




\end{notation}

\section{The algebras $(\A,+,*_M)$}


\subsection{General definitions}

Let $\DO$ be the space of smooth functions with support on a compact subset of $\RE$, and $\DO'$ is its dual (the space of Schwartz distributions). 
As usual, supp $F$ denotes the support of $F \in \DO'$, and sing supp $F$ denotes its singular support.

For every locally integrable function $f \in \L_{\rm loc}^1$ one defines a regular distribution $f_{\DO'}\in {\DO'}$ by
$$
\left\langle f_{\DO'},t\right\rangle=\int_\RE {f(x) t (x)} \, dx  ~, t \in \DO ~.
$$
By abuse of notation, we will usually identify $f_{\DO'}$ with $f$. The $n$th-order Schwartz distributional derivative of the distribution $F$  is defined by
$$\left\langle D_x^n F ,t \right\rangle =(-1)^{n}\left\langle F , d_x^{n}{t} \right\rangle ~, \quad  t\in\DO$$
where $d_x^{n}{t}$ denotes the $n$th-order classical  (pointwise)  derivative of ${t}$.
If $f$ is absolutely continuous, the Schwartz distributional derivative and the classical pointwise derivative (defined a.e.) coincide, i.e.
\begin{equation} \label{1.1}
D_x f_{\DO'} =\left({{{d_x}}f}\right)_{\DO'} ~.
\end{equation}
The dual product of a function $\phi\in \C^\infty$ by a distribution $F \in \DO'$ is defined by
\begin{equation} \label{1.2}
\left\langle \phi F,t\right\rangle=\left\langle  F,\phi t\right\rangle, \quad  t\in \DO
\end{equation}
and it is a generalization of the standard product of functions, i.e. 
\begin{equation} \label{1..2}
\phi (h_{\DO'}) =(\phi h)_{\DO'} ~, \quad \mbox{for all } h \in \L_{\rm loc}^1 ~.
\end{equation}
The dual product is bilinear. Moreover, the distributional derivative $D_x$ satisfies the Leibniz rule with respect to the dual product:
\begin{equation} \label{kkk}
D_x (\phi F) =({d_x}\phi) F+\phi D_x F ~.
\end{equation}

\subsection{The multiplicative product $*$}

 For a detailed presentation and proofs of the main results, the reader should refer to \cite{DP09}. Let $\C^\infty_p$ be the space of piecewise smooth functions on $\RE$: $f\in\C^\infty_p$ iff there is a finite set $I\subset\RE$ such that $f\in\C^\infty(\RE\backslash I)$ and the lateral limits $\lim_{x\rightarrow x_0^\pm}f^{(n)}(x)$ exist and are finite for all $x_0\in I$ and all $n\in \N_0$. We have of course $\C_p^\infty \subset \L_{\rm loc}^1$.

\begin{definition} \label{espaco}
Let $\A$ be the space of piecewise smooth functions $\C^\infty_p$ (regarded as regular distributions) together with their distributional derivatives to all orders.
\end{definition} 

All the elements of $\A$ are distributions with finite singular supports. They can be written explicitly in the form:

\begin{lemma}\label{f6}
$F \in \A$ iff there is a finite set ${ I}=\{x_1<x_2<...<x_m\}$ associated with a set of open intervals
$\Omega_i=(x_i,x_{i+1})$, $i=0,..,m$ (where $x_0=-\infty$ and $x_{m+1}=+\infty$)  such that:
\begin{equation}\label{FormF}
F= f+\Delta^F
\end{equation}
where $f\in \C_p^\infty$ is of the form ($\chi_{\Omega_i}$ is the characteristic function of $\Omega_i$):
\begin{equation}\label{FormF2} 
 f=\sum_{i=0}^m f_i \chi_{\Omega_i} ~,\quad f_i \in \C^{\infty}
\end{equation}
and $\Delta^F$ has support on a subset of $I$:
\begin{equation}\label{FormFF2} 
 \Delta^F=\sum_{i=1}^m \Delta^F_{x_i}  =\sum_{i=1}^m \sum_{j=0}^n
{F}_{ij}\delta^{(j)}_{x_i} ~ ,\quad { F}_{ij} \in \RE ~.
\end{equation}
We have, of course, sing supp $F
\subseteq { I}$.
\end{lemma}


Let us recall the definition of the H\"ormander product of distributions with non-intersecting singular supports [pag.55,  \cite{Hor83}].

\begin{definition}
Let $F,G\in\DO'$ be two distributions such that $\textrm {sing supp } F$ and $\textrm {sing supp } G$ are finite disjoint sets. Let $\left\{\Omega_i\subset\RE,i=1,\ldots,d\right\}$ be a finite covering of $\RE$ such that, on each open set $\Omega_i$, either F or G is a smooth function. The H\"ormander product of F by G is defined by
$$F \cdot G : (F\cdot G)|_{\Omega_i}=F|_{\Omega_i}G|_{\Omega_i}$$
where $F|_{\Omega_i}$ denotes the restriction of the $F$ to the set $\Omega_i$, and likewise for the other distributions. Moreover, the product  $F|_{\Omega_i}G|_{\Omega_i}$ is the dual product defined in (\ref{1.2}).
\end{definition}

Let us emphasise that the H\"ormander product is well-defined for all $F,G \in \A$ provided that $\textrm {sing supp } F$ and $\textrm {sing supp } G$ are finite disjoint sets. We now extend the H\"ormander product to the case of distributions with intersecting singular supports (see \cite{DP09} for details)

\begin{definition}\label{jin}
Let $F,G\in\A$. The product $*$  is defined  by 
\begin{equation} \label{gy}
F*G=\mathop {\lim }\limits_{\varepsilon  \downarrow 0}F(x) \cdot G(x+\epsilon)
\end{equation}
where the product $F(x)\cdot G(x+\epsilon)$ is the H\"ormander product and the limit is taken in the distributional sense.
\end{definition}

Notice that for sufficiently small $\epsilon >0$, $F(x)$ and $G(x+\epsilon)$ have disjoint singular supports,
hence the H\"ormander product in (\ref{gy}) is well-defined.

\indent

The next theorem provides an explicit formula for $F*G$. Let $F,G \in \A$, let $I=(\text{sing supp F} \, \cup \text{sing supp G})
=\{x_1<..<x_m\}$, and consider the  associated set of open intervals $\Omega_i=(x_i,x_{i+1})$, $i=0,..,m$ (with
$x_0=-\infty$ and $x_{m+1}=+\infty$). Then, in view of Lemma \ref{f6}, $F$ and $G$ can be written in the form:
\begin{equation}\label{187}
F = \sum_{i=0}^{m}f_i\chi_{\Omega_i}+ \sum_{i=1}^m \Delta^F_{x_i} \qquad , \qquad
G = \sum_{i=0}^{m}g_i\chi_{\Omega_i}+ \sum_{i=1}^m \Delta^G_{x_i} 
\end{equation}
where $f_i,g_i \in\C^\infty$ and $\Delta^F_{x_i}=0$ if $x_i \in I \backslash \text{sing supp F}$, and likewise for $\Delta^G_{x_i} $. 
Then we have:

\begin{theorem}\label{2.5}
Let $F,G \in \A$ be written in the form (\ref{187}). Then $F*G $ is given explicitly by
\begin{equation} \label{prodf}
F * G =  \sum_{i=0}^{m}f_i g_i\chi_{\Omega_i}+\sum_{i=1}^m\left[  g_i \Delta^F_{x_i} + f_{i-1} \Delta^G_{x_i}\right] ~.
\end{equation}
\end{theorem}

Finally, the main properties of the product $*$ are summarized in the following Theorem (cf. Theorems {3.16 and 3.18}, \cite{DP09}):

\begin{theorem} \label{impo}
The product $*$ is an inner operation in $\A$, it is associative, distributive, non-commutative and it reproduces the product of continuous functions in $\A\, \cap \,\C$. The distributional derivative $D_x$ is an inner operator in $\A$ and satisfies the Leibniz rule with respect to the product $*$.
\end{theorem}

We conclude that the space $\A$, endowed with the product $*$, is
an associative (but non-commutative) differential algebra of
distributions that satisfies the properties stated in Theorem \ref{DP}. It is, however, not the unique algebra that satisfies these conditions, as we now show.

\subsection{The algebras ($\A,+,*_M$).}



Let $F,G \in \A$ be written in the form:
\begin{equation}\label{FormFG}
F= f+ \sum_{x_i \in I_F} \Delta_{x_i}^F \quad , \quad 
G= g+ \sum_{y_i \in I_G} \Delta_{y_i}^G 
\end{equation}
where $f,g\in\C^\infty_p$, $I_F=$ supp $\Delta^F$ and $I_G=$ supp $\Delta^G$. Then:
\begin{equation}\label{F*G}
F*G= fg + \sum_{y_i \in I_G} f * \Delta_{y_i}^G +  
\sum_{x_i \in I_F} \Delta_{x_i}^F *g
\end{equation}
and, likewise
\begin{equation}\label{G*F}
G*F= fg + \sum_{y_i \in I_G}  \Delta_{y_i}^G *f +  
\sum_{x_i \in I_F} g* \Delta_{x_i}^F ~.
\end{equation}
We can combine both formulas and obtain a slightly more general product (one that acts as $F*G$ on the points that belong to a given set $M \subseteq \RE$, and as $G*F$ on the points that don't belong to $M$): 

\begin{definition} \label{Gprod1}
Let $M \subseteq \RE$ and $F,G \in \A$. The product $*_M$ is defined by
 \begin{eqnarray} \label{Gprod11}
F*_M G & = & f \, g + \sum_{y_i \in I_G \cap M} f * \Delta_{y_i}^G +
 \sum_{x_i \in I_F \cap M} \Delta_{x_i}^F *g \\
 \nonumber & &+ \sum_{y_i \in I_G \backslash M}  \Delta_{y_i}^G *f 
+ \sum_{x_i \in I_F\backslash M} g * \Delta_{x_i}^F  
\end {eqnarray}
\end{definition}

Notice that for $M = \RE$ we have $F*_M G=F*G$ and for $M=\emptyset $, $F*_MG=G*F$. The next Remark provides some explicit formulas: 

\begin{remark}\label{2.6}
 Let $n,m \in \N_0$ and $s,t \in \RE$. Let $M \subseteq \RE$ and let $H$ be the Heaviside step function. It follows from (\ref{prodf}) and (\ref{Gprod11}) that:
$$
H(x-t) *_M  \delta_s^{(n)} = \delta_s ^{(n)} *_M H(x-t)=
\left\{ \begin{array}{l} 
\begin{tabular}{ll}
$\delta_s^{(n)}$  & \text{ if } $s>t$\\
0  & \text{ if } $s<t$
\end{tabular}
\end{array} \right.
$$
$$
\delta_t^{(n)} *_M  H(x-t) ~ = ~ H(t-x) *_M  
\delta_t^{(n)} ~ = ~ \chi_M(t) ~\delta_t^{(n)}
$$
$$
H(x-t) *_M  \delta_t ^{(n)} ~ = ~ \delta_t^{(n)} *_M  H(t-x) ~ = ~ (1-\chi_M(t))~ \delta_t^{(n)} 
$$
\begin{equation*}
\delta^{(n)}_s *_M \delta^{(m)}_t=0 ~.
\end{equation*}
\end{remark}


Let us introduce the following distributions:
\begin{definition}
Let $M \subseteq \RE$ and let $F \in \A$ be written in the form (\ref{FormFG}). 
The distribution $F_M\in\A$ associated with $F$ is defined by
\begin{equation}\label{rii}
F_M=\frac{\sqrt{2}}{2} f+\sqrt{2}~ \Xi^{F_M}   \quad \mbox{where} \quad  \Xi^{F_M}=\smashoperator{\sum_{x_i \in I_F \cap M}}  \Delta_{x_i}^F ~.
\end{equation}
\end{definition}

We can now write $F*_M G$ in a compact form: 

\begin{lemma} 
Let $F,G \in \A$ and let $F_M,G_M$ be the associated distributions of the form (\ref{rii}). Then
\begin{equation}\label{hh}
F*_M G=F_M*G_M + G_{\RE \backslash M} * F_{\RE \backslash M} ~.
\end{equation}
\end{lemma} 
\begin{proof}
Using (\ref{rii}), we have
\begin{equation*}
F_M*G_M= \frac{1}{2}fg+ f * \Xi^{G_M} +
\Xi^{F_M} *g
\end{equation*}
and likewise:
\begin{equation*}
G_{\RE \backslash M}*F_{\RE \backslash M}= \frac{1}{2}fg + \Xi^{G_{\RE \backslash M}} *f
+  g * \Xi^{F_{\RE \backslash M}} ~.
\end{equation*}
Hence, (cf. (\ref{Gprod11})): 
\begin{equation*}
F_M*G_M + G_{\RE \backslash M} * F_{\RE \backslash M}=F*_M G ~.
\end{equation*}
\end{proof}

We now study the main properties of $*_M$:

\begin{theorem} 
For all $M\subseteq \RE$, the product $*_M$ is (i) an inner operation in $\A$, (ii) distributive and (iii) associative. Moreover, (iv) it reproduces the usual product of continuous functions in $\A \cap \C$, and (v) the dual product of smooth functions by distributions in $\A$. (vi) It is also partially continuous at zero and (vii) $D_x$ satisfies the Leibniz rule with respect to $*_M$. 
\end{theorem}

\begin{proof}
Let $F,G,J \in \A$ and let $F_M,G_M,J_M$ be the associated distributions defined by (\ref{rii}). Then:\\

(i) By Lemma \ref{f6} we have $F_M,G_M,F_{\RE \backslash M},G_{\RE \backslash M}\in\A$. Since ($\A,+,*$) is an algebra it follows that $F_M*G_M + G_{\RE \backslash M} * F_{\RE \backslash M}\in \A$ and therefore $F*_M G \in \A$.\\

(ii)  From (\ref{hh}) we have:
\begin{equation*}
F*_M J+G*_MJ=F_M*J_M + J_{\RE \backslash M} * F_{\RE \backslash M}+G_M*J_M + J_{\RE \backslash M} * G_{\RE \backslash M} ~.
\end{equation*}
Since $*$ is distributive and $F_M+G_M=(F+G)_M$ for all $M\subseteq\RE$, we get
\begin{align*}
F*_M J+G*_MJ&=(F+G)_M*J_{ M}  + J_{\RE \backslash M} * (F+G)_{\RE \backslash M}\\
&=(F+G)*_MJ
\end{align*}
which proves that the product is right-distributive. Equivalently, one proves that it is also left-distributive.\\

(iii) We have for $F,G,J\in\A$
$$(F*_M G)*_M J=(F*_MG)_M*J_M+J_{\RE \backslash M}*(F*_MG)_{\RE \backslash M}$$
Since
$$(F*_M G)_M=\frac{\sqrt{2}}{2} fg +\sqrt{2}\left( \Xi^{F_M} *g+ f*\Xi^{G_M}\right) $$
and
$$(F*_M G)_{\RE \backslash M}=\frac{\sqrt{2}}{2} fg +\sqrt{2} \left(g*\Xi^{F_{\RE \backslash M}}+ \Xi^{G_{\RE \backslash M}} *f \right) $$
a simple calculation shows that:
\begin{align*}
(F*_M G)*_MJ  = fgj &+ (f*g) * \Xi^{J_M}
+(f*\Xi^{G_M})*j +(\Xi^{F_M} *g)*j\\
&+\Xi^{J_{\RE \backslash M}}*(f*g)+j* (g*\Xi^{F_{\RE \backslash M}})+j*(\Xi^{G_{\RE \backslash M}}*f)   
\end{align*}
Using the associativity of $*$ we get
\begin{align*}
(F*_M G)*_MJ= fgj &+ f *(g*\Xi^{J_M})
+f*(\Xi^{G_M}*j)+\Xi^{F_M} *(g*j)\\
&+(\Xi^{J_{\RE \backslash M}}*f)*g 
+(j*\Xi^{G_{\RE \backslash M}})*f
+(j* g)*\Xi^{F_{\RE \backslash M}}   
\end{align*}
which is exactly $F*_M(G*_MJ)$. Hence the product $*_M$ is associative.\\

(iv)  If $F,G\in(\A\cap\C)$ then $F=f$ and $G=g$ with $f,g \in \C \cap \C_p^\infty$ (cf. (\ref{FormF})). Hence, from (\ref{Gprod11}), $F*_M G=fg$.\\

(v) If $F \in C^\infty$ then in (\ref{FormFG}) we have $F=f$. It follows from (\ref{prodf}) that $F*G=G*F=fG$ for all $G\in \A$ and so, from (\ref{Gprod11}), that $F*_MG=G*_MF=fG$.\\

(vi) Since $\xi *_M F= F*_M \xi= \xi F$ for all $\xi \in \C^\infty$ and $F\in \A$, the partial continuity of $*_M$ follows from the same property for the dual product. Let then $(F_n)_{n\in \N}$, $F \overset{\DO'}{\longrightarrow} 0$ be a sequence in $\A$. We have:
$$
\lim_{n\to +\infty} \langle \xi F_n,t \rangle =
\lim_{n\to +\infty} \langle F_n, \xi t\rangle =0\, , \quad \forall t \in \DO 
$$
and so $\xi F_n \overset{\DO'}{\longrightarrow} 0$.\\
 
(vii) Let us write $F*_M G$ in the form
\begin{equation}\label{F*_MG2}
F*_M G=F * G+\frac{\sqrt{2}}{2} J_{\RE \backslash M}
\end{equation}
where $J= G*F-F*G$,
and $J_{\RE \backslash M}$ is defined by (\ref{rii}). Notice that from (\ref{F*G}) and (\ref{G*F}) we have explicitly:
$$
\frac{\sqrt{2}}{2} J_{\RE \backslash M}= \sum_{y_i \in I_G \backslash M} \left(\Delta_{y_i}^G * f -f * \Delta_{y_i}^G \right) + \sum_{x_i \in I_F \backslash M} \left(g* \Delta_{x_i}^F  - \Delta_{x_i}^F * g \right) ~.
$$
Moreover, since $J$ is of finite support (supp $J \subseteq I_F \cup I_G$), we have
$$
D_x \left(J_{\RE \backslash M}\right) = \left(D_x J\right)_{\RE \backslash M} ~.
$$
It follows that
$$
D_x\left( F*_M G \right) = D_x\left( F * G \right) + \frac{\sqrt{2}}{2} \left( D_x(G*F) - D_x(F*G) \right)_{\RE \backslash M}
$$
and since $D_x$ satisfies the Leibniz rule with respect to the product $*$, by using (\ref{F*_MG2}) the terms on the r.h.s can be easily shown to yield 
$(D_x F) *_M G+F*_M (D_x G)$, concluding the proof.\\
\end{proof}

Theorem \ref{DP} in the Introduction is a simple corollary of the previous result.

\section{Main Theorem} 

In this section we prove the Main Theorem. 
Several preparatory results that are required for the proof will be given in section 3.1  (Theorems \ref{ak}, \ref{ii}, \ref{3.1} and \ref{ty}).  

\subsection{Preparatory results}

\begin{theorem}\label{ak}
Let $\xi \in \C^\infty$ and $F\in\G$. Then 
\begin{equation}\label{uo}
\xi\circledast F=\xi F=F\circledast \xi ~.
\end{equation}
\end{theorem}

\begin{proof}
If $F\in \G \cap \C$ then from (A4)
\begin{equation*}
\xi\circledast F=\xi F ~.
\end{equation*}
Moreover, if (\ref{uo}) is valid for some $F\in\G $ (and all $\xi \in \C^\infty$) then it is valid for $F'=D_xF$:
 \begin{eqnarray*}
(\xi\circledast F)'=(\xi F)'&\Longleftrightarrow &\xi'\circledast F+\xi\circledast F'=\xi'F+\xi F'\\
&\Longleftrightarrow & \xi\circledast F'=\xi F'
\end{eqnarray*}
since $\xi'\circledast F=\xi'F$ by (\ref{uo}), and the dual product satisfies the Leibniz rule.
This proves that 
\begin{equation}\label{jj}
\xi\circledast g^{(k)} =\xi g^{(k)}
\end{equation}
for all $g\in \G\cap\C$, $\xi \in \C^\infty$ and all $k\in\N _0$.  

We now extend the previous result to all $F\in \G$. We will need to impose the extra conditions (A5.1) and (A5.2).
Let $(\phi_i)_i$ be a countable family of smooth real functions satisfying:
\begin{enumerate}

\item [(P1)] $\sum_{i=1}^{+\infty}\phi_i(x)=1,\,\forall x\in\RE$.
	
\item [(P2)] $\text{supp } \phi_i$ is compact.
  
\item [(P3)] If $\Xi \subset \RE$ is compact then $\Xi \cap \text{supp } \phi_i\neq \emptyset$ only for a finite number of functions $\phi_i$.
\end{enumerate}
Then, $\forall F\in \G$ we have
$$
F=(\sum_{i=1}^{+\infty}\phi_i)F=\sum_{i=1}^{+\infty}\phi_i F=\sum_{i=1}^{+\infty}\phi_i g_i^{(k_i)}
$$
where (cf. (A5.1)) $k_i\in \N_0$ and $g_i\in \C\cap \G$ is such that for some bounded open set $\Omega_i \supset$ supp $\phi_i$ we have $F|_{\Omega_ {i}}=g^{(k_i)}_i|_{\Omega_ {i}}$.
Using (\ref{jj}) we get:
$$
F=\sum_{i=1}^{+\infty}\phi_i\circledast g_i^{(k_i)} ~.
$$
Let $F_i=\phi_i\circledast g_i^{(k_i)}$. Then $F_i\in \G$ and is of compact support (because $\phi_i\circledast g_i^{(k_i)}=\phi_i g_i^{(k_i)}$ and $\text{supp } \phi_i$ is compact).
Hence $F_i=h_i^{(s_i)}$ for some $h_i\in \G \cap \C$, $s_i \in \N_0$ (cf. (A5.1)), and thus, for every $F\in\G$, $\xi \in \C^\infty $ and $n \in \N$:
 \begin{equation} \label{f*F0} 
 \xi\circledast F  =  \xi \circledast  \sum_{i=1}^{+\infty} h_i^{(s_i)} = \xi \circledast  \sum_{i=1}^{n} h_i^{(s_i)} + \xi \circledast  \sum_{i=n+1}^{+\infty} h_i^{(s_i)}~.
 \end{equation}
 Let $z_n=\sum_{i=n+1}^{+\infty} h_i^{(s_i)}$. Then
 $
 z_n=  F - \sum_{i=1}^{n} h_i^{(s_i)} \in \G
 $, and 
 $z_n \overset{\DO'}{\longrightarrow} 0$ (since $F=\lim_{n\to +\infty} \sum_{i=1}^{n} h_i^{(s_i)}$ in $\DO'$). 
We then rewrite (\ref{f*F0}) in the form
$$
\xi\circledast F  = \xi \circledast  z_n+  \xi \cp \sum_{i=1}^{n}  h_i^{(s_i)} = \xi \circledast  z_n+  \sum_{i=1}^{n} \xi \cp h_i^{(s_i)}
$$
and take the limit $n \to +\infty$. Using the partial continuity of the product (cf. (A5.2)), we find:
\begin{equation}\label{f*F}
\xi \cp F = \lim_{n\to +\infty} \sum_{i=1}^{n} \xi \cp h_i^{(s_i)} 
=\sum_{i=1}^{+\infty} \xi \, h_i^{(s_i)}  
 = \xi \sum_{i=1}^{+\infty}h_i^{(s_i)} =\xi \, F \nonumber
\end{equation}
where we used (\ref{jj}) in the second step.
In the same manner one proves that $F\circledast \xi=\xi F$.

\end{proof}

\begin{remark}\label{remarkA6}
Notice that if we assume that $\G$ satisfies (A6), the proof of the previous Theorem is concluded in eq.(\ref{jj}) since every $F \in \G$ is (globally) a finite order derivative of some $g \in \G \cap \C$. The condition (A6) is satisfied by $\G=\A$, and also by $\G = \DO'(\Omega)$ for $\Omega \subseteq \RE$ a compact set. If $\G$ satisfies (A6) we can also conclude that $\cp$ is partially continuous (i.e. it also satisfies (A5.2)). This follows from the partial continuity of the dual product and the fact that from (\ref{jj}), $\xi \cp F= \xi F$ for all $\xi \in \C^\infty$ and $F\in \G$.    
\end{remark}

\begin{remark}\label{remarkak}
Theorem \ref{ak} still holds if in the definition of $\G$ the condition (A5.1) is replaced by (A5.1'). To prove this let us consider the partition of unity $(\phi_i)_i$ defined above by (P1)-(P3). For $F \in \G$ we have
$$
F=(\sum_{i=1}^{+\infty} \phi_i)F=\sum_{i=1}^{+\infty}\phi_i F ~.
$$
Let $F_i=\phi_i F$; since $F_i$ is of compact support and $F_i\in\G$ (cf. (A5.1')) we have $F_i=h_i^{(s_i)}$ for some $s_i \in \N_0$ and $h_i \in \C$ (cf. Corollary 3.4-2a, \cite{Zemanian}). Moreover, since anti-differentiation is an inner operation in $\G$ (cf. (A5.1')), we also have $h_i \in \G$. Hence $F\in\G$ can be written in the form:
$$
F=\sum_{i=1}^{+\infty} h_i^{(s_i)} ~, \quad h_i\in \C\cap\G ~,\quad s_i\in\N_0 
$$
and the rest of the proof follows from eq.(\ref{f*F0}).
\end{remark}

\begin{lemma}\label{cris}
Let $F,G \in\G$. Then  $\text{supp}~(F\circledast G)\subseteq \text{supp }F\cap\text{supp }G$.
\end{lemma}
\begin{proof}
We will prove that $\text{supp }(F\circledast G) \subseteq \text{supp } F$ (the same result is valid for $G$).
The previous statement is equivalent to proving that ($\Omega^c$ denotes the complement of $\Omega$):
\begin{equation}\label{B1}
\left\langle  F\circledast G , t\right\rangle=0 ~, \quad \forall t \in \DO: \text{supp }t \subset (\text{supp } F)^c
\end{equation}
with the obvious exception of the case $\text{supp } F=\RE$ for which the result is trivial. Let us consider a partition of unity $\phi_1, \phi_2\in\C^{\infty} $ such that:
\begin{enumerate}
	\item[(i)] $\phi_1 + \phi_2 =1$,
	\item[(ii)] $\phi_1(x\in \text{supp } F) =0$,
	\item[(iii)] $\phi_2(x\in \text{supp } t) =0$~.
\end{enumerate}
Then
 \begin{eqnarray*}
\left\langle  F\circledast G , t\right\rangle &=&\left\langle (\phi_1 + \phi_2) \circledast (F \circledast G) , t\right\rangle\\
&=& \left\langle  (\phi_1 \circledast F)\circledast G , t\right\rangle +  \left\langle  \phi_2 \, (F\circledast G) , t\right\rangle
\end{eqnarray*}
where we used (\ref{uo}) to obtain the second term. It follows that:
\begin{eqnarray*}
\left\langle  F\circledast G , t\right\rangle =\left\langle (\phi_1 F)\circledast G , t\right\rangle + \left\langle F\circledast G , \phi_2 t\right\rangle =0
\end{eqnarray*}
because, by (ii) and (iii), $\phi_1 F=0$ and $\phi_2 t=0$, respectively.
\end{proof}

\begin{theorem}\label{ii}
Let $s,t\in \RE$. Then $H(x-s)\circledast H(x-t) = H(x-max\left\{s,t\right\}).$
\end{theorem}
\begin{proof}
Let us first consider $s<t$. Since $\text{supp }H(s-x)\cap\text{supp } H(x-t)= \emptyset$, it follows from {Lemma} \ref{cris} that 
\begin{equation}\label{999}
H(s-x)\circledast H(x-t) =0 ~.
\end{equation}
 Taking into account that $H(s-x) =1-H(x-s) $, we easily obtain $H(x-s)\circledast H(x-t) =H(x-t)$. The other case $s>t$ is proved in the same way.

Let now $s=t$.
It follows from (A4) that
\begin{equation} \label{7}
\left|x-s\right|\circledast \left|x-s\right|=(x-s)^2
\end{equation}
Twice differentiating this equation, we get
\begin{equation} \label{8}
\delta_s \circledast \left|x-s\right|+ (D_x\left|x-s\right|)\circledast(D_x\left|x-s\right|)+\left|x-s\right|\circledast \delta_s  =1
\end{equation}
where we took into account that
 \begin{equation} \label{8h}
D_x^2\left|x-s\right|=2\delta_s ~.
\end{equation}
We now prove that $\delta_s \circledast \left|x-s\right|=0$. To make it simple let $s=0$. We have: 
\begin{eqnarray*}
\delta \circledast \left|x\right| &=& \delta \circledast \left( x \circledast H -x \circledast H(-x) \right)  \\
&=& (\delta \circledast x)\circledast H-(\delta \circledast x)\circledast H(-x)=0 ~.
\end{eqnarray*}
Notice that $x \circledast H=x H$ and $\delta\circledast x=x\circledast \delta= x \delta= 0$ (cf. Theorem \ref{ak}).
In the same way one proves that ${\left|x-s\right|}\circledast \delta_s =0 $. Hence eq.(\ref{8}) reduces to
\begin{eqnarray*}
 & &({2H(x-s)-1})\circledast ({2H(x-s)-1})={1}\\
&\Longleftrightarrow & {H(x-s)}\circledast{H(x-s)}={H(x-s)}
\end{eqnarray*}
which concludes the proof.
\end{proof}

\indent 

Before we proceed to the next Theorem, let us recall the following useful formula which is valid for all $n,m \in \N_0$ (eq.(26), section 2.6, \cite{Kanwal}):
\begin{equation}\label{io}
x^{n}\delta ^{(m)}=\left\{ 
\begin{array}{lll}
0~, && \quad m < n \\	
(-1)^n\frac{m!}{(m-n)!}\delta ^{(m-n)}~, && \quad m \ge n
\end{array}
\right.
\end{equation}
where we used the convention $0!=1$. 
The case $m\ge n$ can be easily inverted, yielding
\begin{equation}\label{iok}
\delta ^{(j)}=(-1)^n\frac{j!}{(j+n)!} x^{n}\delta ^{(j+n)} ~, \quad \forall j,n \in \N_0 ~.
\end{equation}


\begin{theorem}\label{3.1} For every $s,t \in \RE$, and every $i,j \in \N_0$ we have
\begin{equation}\label{10}
\delta^{(i)}_s \circledast \delta^{(j)}_t=0 ~.
\end{equation}
\end{theorem}

\begin{proof}
Since $\text{supp }(\delta^{(i)}_s \circledast \delta^{(j)}_t)\subseteq \text{supp }\delta^{(i)}_s\cap\text{ supp }\delta^{(j)}_t$
(cf. Lemma \ref{cris}), for $t \not= s$ we have
$\delta^{(i)}_s \circledast \delta^{(j)}_t=0$.
Moreover, if $t=s$, we get:
\begin{equation}\label{kj}
\delta^{(i)}_s \circledast \delta^{(j)}_s=\sum_{k=0}^n a_k\delta^{(k)}_s 
\end{equation}
for some $n \in \N_0$ and $a_k\in\RE$, $k=0,..,n$. To simplify the presentation, assume that $s=0$ and $i\leq j$. Then:
\begin{eqnarray*}
 x^{i+1}\circledast(\delta^{(i)}\circledast\delta^{(j)})  & = & (x^{i+1}\circledast\delta^{(i)})\circledast\delta^{(j)}\\
&= & (x^{i+1}\delta^{(i)})\circledast\delta^{(j)}=0
\end{eqnarray*}
where we used (\ref{io}).
Substituting (\ref{kj}) in the previous expression: 
\begin{equation}
x^{i+1}\circledast\sum_{k=0}^n a_k\delta^{(k)}= 
x^{i+1}\sum_{k=0}^n a_k\delta^{(k)}=0 ~\Longrightarrow ~
a_k=0, \quad \forall k\geq i+1
\end{equation}
and so
\begin{equation}\label{D}
\delta^{(i)}\circledast\delta^{(j)}= \sum_{k\leq i} a_k \delta ^{(k)} ~.
\end{equation}
Let us return to (\ref{iok}) and set $n=i+1$:
\begin{equation}
\delta ^{(j)}=b_{ij}\delta^{(j+i+1)}x^{i+1} \quad {\rm where} \quad b_{ij}=(-1)^{i+1} \tfrac{j!}{(j+i+1)!} ~.
\end{equation}
It follows that:
\begin{equation}\label{2}
\delta ^{(i)}\circledast \delta ^{(j)}=b_{ij} \left( \delta^{(i)}\circledast\delta^{(j+i+1)} \right) \circledast x^{i+1} 
\end{equation}
where we used (\ref{uo}) and the associativity of $\circledast$. Since (\ref{D}) is valid for all $j$, we have:
\begin{equation}\label{1}
\delta ^{(i)}\circledast \delta ^{(j+i+1)}=  \sum_{k\leq i} a'_k \delta ^{(k)}
\end{equation}
for some $a'_k\in\RE$. Substituting in (\ref {2}) and using (\ref{io}), we finally get
\begin{eqnarray*}
\delta^{(i)}\circledast\delta^{(j)} = b_{ij} \left( \sum_{k\leq i} a'_k \, \delta ^{(k)}\right) \circledast x^{i+1}
= b_{ij} \sum_{k\leq i} a'_k \, \delta ^{(k)} x^{i+1}=0
\end{eqnarray*}
which concludes the proof.
\end{proof}

\begin{theorem} \label{ty}
\begin{flushleft}
For all $n\in \N_0$ and $s,t\in \RE$ we have:
\begin{itemize}
	 \item [(1)]
 $
  H(x-t)\circledast \delta_s^{(n)} = \delta_s ^{(n)} \circledast H(x-t)=
\left\{ \begin{array}{l} 
\begin{tabular}{ll}
$\delta_s^{(n)}$ & \text{ if } $s>t$\\
0  & \text{ if } $s<t$
\end{tabular}
\end{array} \right.
$
	  \item [(2)]$
H(x-t) \circledast \delta_t ^{(n)}=c_t\delta_t^{(n)},\,\,\,\,\,\,\,
\delta_t ^{(n)} \circledast H(x-t)=(1-c_t)\delta_t^{(n)}$
where $c_t$ is some function $c_t:\RE\longrightarrow \left\{0,1\right\}$.

\end{itemize}
\end{flushleft}
\end{theorem}

\begin{proof}
(1) For $s<t$, $\text{supp }H(x-t)\cap\text{supp } \delta _s^{(n)}= \emptyset$ and thus the product is zero.
For $s>t$, we have $H(x-t)=1-H(t-x)$ and since  $\text{supp }H(t-x)\cap\text{supp } \delta _s^{(n)}= \emptyset$ we get: 
$$
H(x-t)\circledast\delta_s^{(n)}=\delta_s^{(n)}\circledast H(x-t)=\delta_s^{(n)} ~.
$$
(2) Let us begin by proving that the formulas are true for $n=0$. Assume for simplicity that $t=0$.
Since $\text{supp }H(x)\cap\text{supp }\delta =\left\{0\right\}$ we have
$H(x)\circledast\delta={\sum_{k=0}^m} c_k \delta ^{(k)}$, for some $m \in \N_0$ and $c_k\in\RE$.
As before 
$$(H\circledast\delta)\circledast x=H\circledast(\delta\circledast x)=0 ~\Longrightarrow ~ \sum_{k=0}^m c_k (\delta ^{(k)}\circledast x)=0 ~\Longrightarrow ~ c_k=0, ~ \forall k\neq0
$$
and thus $H\circledast\delta=c\delta$ for some $c\in \RE$. Moreover (cf. Theorem \ref{ii})
\begin{eqnarray*}
&& H\circledast(H\circledast\delta)=(H\circledast H)\circledast\delta=H\circledast\delta \\
&\Longleftrightarrow &  H\circledast c\delta=c\delta  ~\Longleftrightarrow ~c^2\delta=c\delta ~\Longleftrightarrow ~  c=0 \vee c=1 ~.
\end{eqnarray*}
Generalizing now for $t\in\RE$:
$H(x-t)\circledast\delta_t=c_t \delta_t$
where $c_t$ is an arbitrary function $c_t:\RE\longrightarrow \left\{0,1\right\}$.
Each function $c_t$ defines a particular $\circledast$-product. Let us denote it by $\circledast_{c_t}$.
Thus $H(x-t)\circledast_{c_t}\delta_t=c_t\delta_t$. 

Let us proceed.
Differentiating $H(x-t)\circledast_{c_t} H(x-t)=H(x-t)$ we get:
\begin{eqnarray*}
& &\delta_t\circledast_{c_t} H(x-t)+H(x-t)\circledast_{c_t}\delta_t=\delta_t  \\
&\Longleftrightarrow & \delta_t\circledast_{c_t} H(x-t)=(1-c_t)\delta_t  
\end{eqnarray*}
which completes the proof of (2) for $n=0$.

Assume now that
\begin{equation}\label{1t}
H(x-t)\circledast_{c_t}\delta_t^{(n)}=c_t\delta_t^{(n)}
\end{equation}
holds for some $n\in \N$.
Differentiating (\ref{1t}) we get
$$
\delta_t\circledast_{c_t}\delta_t^{(n)}+H(x-t)\circledast_{c_t}\delta_t^{(n+1)}=c_t\delta_t^{(n+1)}
$$
and since the first term is zero (cf. Theorem \ref{3.1}) we conclude that (\ref{1t}) is valid for $n+1$, and thus for all $n \in \N_{0}$.
In the same way one proves that $\delta_t ^{(n)} \circledast H(x-t)=(1-c_t)\delta_t^{(n)}$, for all $n \in \N_{0}$.
\end{proof}\\




\subsection{Main Theorem}

We can now easily prove the Main Theorem.

\begin{proof}[Proof]
	
The inclusion $\A \subseteq \G$ follows directly from (A1') and the fact that the distributional derivative $D_x$ is an inner operator in $\G$. 

Let then $F,G \in \A$, we want to prove that $F \circledast G= F *_M G$ for some $M \subseteq \RE$. Let us write $F,G$ in the form (\ref{FormF}). From the distributive property of $\circledast$:
\begin{eqnarray}\label{proddist}
F\circledast G &=& (f+\Delta^F) \circledast (g+\Delta^G)\\ 
&=& f\circledast g + f\circledast \Delta^G+\Delta^F \circledast g + \Delta^F \circledast \Delta^G \nonumber ~.
\end{eqnarray}

Let us consider each term separately:

\indent

1) We first prove that  
\begin{equation}\label{func_func}
f\circledast g= f*_Mg ~ , \quad  \forall f,g \in \C_p^\infty ~, \quad \forall M\subseteq \RE ~.
\end{equation}
Since (cf. Theorem \ref{ak}):
$$
\xi H(x-a)=\xi \circledast H(x-a) = H(x-a) \circledast \xi \, , \quad \forall \xi \in \C^{\infty} 
$$
and also (cf. Theorem \ref{ii}):
$$
H(x-a) \circledast H(x-b) = H(x-{\rm max}\{a,b\})
$$
we have for $f=f_iH(x-a)$ and $g=g_jH(x-b)$, $f_i,g_j \in \C^\infty$, using the associativity of $\circledast$:
\begin{equation}\label{eqfunc}
f\circledast g = f g= f*_M g
\end{equation}
where the second identity follows from (\ref{Gprod11}) and holds for all $M \subseteq \RE$. Moreover, for $f$ or $g$ smooth, (\ref{eqfunc}) also holds (cf. Theorem \ref{ak}). 
Since every $f,g \in C_p^\infty$ is the sum of a smooth function with a finite number of functions of the form $\xi H(x-a)$, $\xi \in \C^\infty$, $a\in \RE$, using the distributive property of $\circledast$ and $*_M$ we get (\ref{func_func}). 

\indent

2) We now prove that for some $M\subseteq \RE$ 
\begin{equation} \label{eqfuncdist}
f \circledast \Delta^G = f *_M \Delta^G \, , \quad \forall f \in \C_p^\infty ~, ~\forall G \in \A ~.
\end{equation} 
It follows from Theorem \ref{ty} and Remark \ref{2.6} that
\begin{equation} \label{3.24}
H(x-t)\circledast\delta_s^{(n)}=H(x-t)*_M\delta_s^{(n)}
\end{equation}
where $M=\left\{t\in\RE:c_t=0\right\}$.
Since (cf. Theorem 2.11(v) and Theorem 3.1) 
 \begin{equation} \label{xiF}
 \xi\circledast F = \xi *_M F=   \xi F \, , \quad  \forall  \xi \in\C^{\infty} \, , \quad \forall F \in \A
\end{equation}
we get
\begin{eqnarray}
(\xi H(x-t))\circledast\delta_s^{(n)} &=&
\xi \cp (H(x-t)\circledast\delta_s^{(n)}) \\
&=&\xi *_M (H(x-t)*_M\delta_s^{(n)})\nonumber\\
&=&(\xi H(x-t))*_M\delta_s^{(n)} \nonumber
\end{eqnarray}
where in the first and last steps we used (\ref{xiF}) and the associativity of $\cp$ and $*_M$, and in the second step we used (\ref{3.24}) and (\ref{xiF}).

Moreover, $\Delta^G$ is of the form (\ref{FormFF2}) and every $f\in\C_p^{\infty}$ is the sum of some $\psi \in \C^\infty$ with a finite linear combination of terms of the form $\xi H(x-t),~ \xi\in \C^{\infty}$. Hence,
using the distributive property of $\circledast$ and $*_M$, we get (\ref{eqfuncdist}).
An equivalent proof shows that
\begin{equation} \label{distri_funct}
\Delta^F \circledast g =\Delta^F*_M g ~, \quad \forall F\in\A \, , \quad  \forall g\in \C_p^{\infty} ~.
\end{equation}
Notice that once the set $M$ is fixed for the product  $g \circledast \Delta^F$, the product in the reversed order is also fixed, i.e. $\Delta^F \circledast g= \Delta^F *_M g$ (with the same $M$); this follows from Theorem \ref{ty}(2), which determines that the function $c_t$ (and thus the set $M$) is the same in both products.

\indent

3) Finally, we have 
\begin{equation}\label{di_di}
\Delta^F \circledast \Delta^G= \Delta^F *_M \Delta^G =0
\end{equation}
which is true because $\Delta^F$ and $\Delta^G$ are both of the form (\ref{FormFF2}) and so, from  the distributive property of $\circledast$, Theorem \ref{3.1} and the definition of 
$*_M$ (\ref{Gprod11}), we easily conclude that both products in (\ref{di_di}) are zero.     

\indent

Substituting (\ref{func_func},\ref{eqfuncdist},\ref{distri_funct},\ref{di_di}) in (\ref{proddist}) we obtain $F \circledast G=F*_M G$, concluding the proof.

\end{proof}

\subsection*{Acknowledgments}

The authors would like to thank the anonymous referee from Monatshefte f\"ur Mathematik for pointing out an important flaw in the proof of Theorem 3.1. 

Cristina Jorge was supported
by the PhD grant SFRH/BD/85839/2014 of the Portuguese Science
Foundation.



\vspace{4cm}

\textbf{Author's addresses:}

\begin{itemize}
\item \textbf{Nuno Costa Dias}\footnote{Corresponding author} and \textbf{Jo\~ao Nuno Prata:
}Escola Superior N\'autica Infante D. Henrique. Av. Eng.
Bonneville Franco, 2770-058 Pa\c{c}o d'Arcos, Portugal and Grupo
de F\'{\i}sica Matem\'{a}tica, Universidade de Lisboa, Av. Prof.
Gama Pinto 2, 1649-003 Lisboa, Portugal

\item \textbf{Cristina Jorge}: Departamento de Matem\'{a}tica.
Universidade Lus\'{o}fona de Humanidades e Tecnologias. Av. Campo
Grande, 376, 1749-024 Lisboa, Portugal and Grupo de F\'{\i}sica
Matem\'{a}tica, Universidade de Lisboa, Av. Prof. Gama Pinto 2,
1649-003 Lisboa, Portugal

\end{itemize}

\vspace{-0.1cm}

\small

{\it E-mail address} (NCD): ncdias@meo.pt

{\it E-mail address} (CJ): cristina.goncalves.jorge@gmail.com

{\it E-mail address} (JNP): joao.prata@mail.telepac.pt

\end{document}